\documentclass[10pt]{article}
\usepackage[colorlinks]{hyperref}
\AtBeginDocument{
	\hypersetup{
		colorlinks=true,
 		linkcolor=RoyalBlue,
  		citecolor=RedOrange,
	 }
}
\usepackage[usenames,dvipsnames]{xcolor}
\usepackage[top = 1in, bottom = 1in, left = .75in, right =.75in]{geometry}
\usepackage{amsthm,amsmath,amssymb}
\usepackage{graphicx}
\usepackage{verbatim}
\usepackage{enumerate}
\usepackage{bbm}
\usepackage{xcolor}

\theoremstyle{plain}
\newtheorem{theorem}{Theorem}
\newtheorem{lemma}[theorem]{Lemma}
\newtheorem{corollary}[theorem]{Corollary}

\theoremstyle{definition}

\newtheorem{example}[theorem]{Example}
\newtheorem*{example*}{Example}

\theoremstyle{remark}
\newtheorem{remark}[theorem]{Remark}

\numberwithin{equation}{section}

\newcommand{\be}{\begin{equation}}
\newcommand{\ee}{\end{equation}}
\newcommand{\mean}{\mathbb{E}}
\newcommand{\pr}{\mathbb{P}}
\newcommand{\var}{\mathrm{var}}
\newcommand{\ind}[1]{\mathbbm{1}_{#1}}
\newcommand{\Bin}[2]{\operatorname{Bin}(#1,#2)}

\newcommand{\lp}{\left(}
\newcommand{\rp}{\right)}

\newcommand{\lf}{\lfloor}
\newcommand{\rf}{\rfloor}
\newcommand{\lc}{\lceil}
\newcommand{\rc}{\rceil}

\newcommand{\bs}{\boldsymbol}
\newcommand{\bw}{\boldsymbol w}
\newcommand{\bt}{\boldsymbol \tau}

\newcommand{\vecw}[2]{\ensuremath{w_{#1},\dots,w_{#2}}}

\newcommand{\sper}[1]{\mathcal{S}_{#1}} 			
\newcommand{\stp}[1]{\mathcal{C}_{#1}} 			
\newcommand{\spt}[1]{\mathcal{GC}_{#1}}			
\newcommand{\hd}{H_{n}} 						
\newcommand{\leaves}[1]{L_{#1}} 					
\newcommand{\diam}{\mathrm{diam}(T_n)} 			
\newcommand{\degree}{\operatorname{deg}}		
\newcommand{\tail}[1]{LT_{#1}} 					

\def\E{\Bbb E}
\def\P{\Bbb P}

\def\G{\Gamma}
\def\eps{\varepsilon}
\newcommand{\cov}{\mathrm{cov}}

\title{On random trees obtained from permutation graphs}
\author{H\"{u}seyin Acan\\Department of Mathematics\\Rutgers University\\Piscataway, NJ 08854\\USA\\
\texttt{huseyin.acan@rutgers.edu}
\and
Pawe{\l} Hitczenko\\
Department of Mathematics\\ Drexel University\\ Philadelphia, PA 19104 \\ USA\\ \texttt{phitczen@math.drexel.edu}
}
\date{}

\begin{document}

\maketitle

\begin{abstract}
A permutation $\boldsymbol w$ gives rise to a graph $G_{\boldsymbol w}$; the vertices of $G_{\boldsymbol w}$ are the letters in the permutation and the edges of $G_{\boldsymbol w}$ are the inversions of $\boldsymbol w$. We find that the number of trees among permutation graphs with $n$ vertices is $2^{n-2}$ for $n\ge 2$. We then study $T_n$, a uniformly random tree  from this set of trees. In particular, we study the number of vertices of a given degree in $T_n$, the maximum degree in $T_n$, the diameter of $T_n$, and the domination number of $T_n$. Denoting the number of degree-$k$ vertices in $T_n$ by $D_k$, we find that $(D_1,\dots,D_m)$ converges to a normal distribution for any fixed $m$ as $n\to \infty$. The vertex domination number of $T_n$ is also asymptotically normally distributed as $n\to \infty$. The diameter of $T_n$ shifted by $-2$ is binomially distributed with parameters $n-3$ and $1/2$. Finally, we find the asymptotic distribution of the maximum degree in $T_n$, which is concentrated around  $\log_2n$.

\end{abstract}

\section{Introduction} \label{sec: Intro}

A permutation graph is an undirected graph obtained from a permutation by drawing an edge for each inversion in the permutation. For a permutation $\vecw{1}{n}$ of $n$ numbers, a pair $(w_a,w_b)$ is an \emph{inversion}, if $a<b$ and $w_a>w_b$.  Thus, formally, given a permutation $\bw=\vecw{1}{n}$ of $[n]:=\{1,\dots,n\}$, the permutation graph $G_{\bw}$ is defined to be the (undirected) graph with the vertex set $[n]$ and the edge set $\{(w_a,w_b): (w_a,w_b) \text{ is an inversion}\}$. It follows from the definition that the number of edges in $G_{\bw}$ is the same as the number of inversions of $\bw$. Since a permutation is uniquely determined by the set of its inversions, two different permutations yield two different graphs. Consequently, we have $n!$ permutation graphs on the vertex set $[n]$, as opposed to $2^{{n \choose 2}}$ general graphs. Note that our definition of a permutation graph was given by Even et al.~\cite{ELP} and it is different from the one given by Chartrand and Harary \cite{CH}.

Permutation graphs form a subclass of perfect graphs and hence various NP-complete problems in general graphs, including the coloring problem, the maximum clique problem, and the maximum independent set problem, have polynomial time solutions in permutation graphs. This aspect of permutation graphs has led to many studies that are computational in nature. Frequently, a problem on permutation graphs can be easily translated to a problem on permutations. For example, a clique in a permutation graph corresponds to a decreasing subsequence in the accompanying permutation and likewise an independent set corresponds to an increasing subsequence. Permutation graphs have been also used in the theory of limits of combinatorial structures to relate graphons (i.e. graph limits) with permutons (limits of permutations). We refer the reader to \cite{GGKK} for more information on these concepts and relation between them.

A variant of permutation graphs with an additional parameter $t$ was studied by Adin and Roichman~\cite{AR} and by Keevash et al.~\cite{KLS}. For $t=0$, Adin and Roichman, and for general $t$, Keevash et al. found the maximum number of edges in $\Gamma_{t,\boldsymbol \pi}$, where $\boldsymbol \pi$ is a permutation of $[n]$ and  $\Gamma_{t,\boldsymbol \pi}$ is the graph obtained from $\boldsymbol \pi$ in the following way. The vertex set of $\Gamma_{t,\boldsymbol \pi}$ is $[n]$, and two vertices $i<j$ are adjacent if $i$ and $j$ form an inversion and there are at most $t$ numbers $k$ with the property $i<k<j$ but $\boldsymbol \pi^{-1}(i)>\boldsymbol \pi^{-1}(k)>\boldsymbol \pi^{-1}(j)$. It follows that the number of edges in $\Gamma_{0,\boldsymbol \pi}$ is the number of permutations that are covered by $\boldsymbol \pi$ in the strong Bruhat order. 

A parallel notion for connected permutation graphs is indecomposable permutations. A permutation $\bw=\vecw{1}{n}$ of $[n]$ is called \emph{decomposable} at $m$ if $\{w_1,\dots,w_m\}= \{1,\dots,m\}$ for $m<n$. If there is no such $m$, then $\bw$ is called \emph{indecomposable}. Koh and Ree \cite{KR} showed that a permutation $\bw$ is indecomposable if and only if the graph $G_{\bw}$ is connected. This connection makes it possible to transfer various results and relations about indecomposable permutations to connected permutation graphs. For instance, the bijection between indecomposable permutations of $[n]$ and pointed hypermaps of size $n-1$, given by Mendez and Rosenstiehl ~\cite{MR}, implies a one--to--one correspondence between connected permutation graphs and pointed hypermaps of size $n-1$. For more information on indecomposable permutations, we refer the reader to B\'{o}na ~\cite{Bona04} or Flajolet and Sedgewick ~\cite{FS}. 

Indecomposable permutations were first studied by Lentin ~\cite{Lentin} and Comtet ~\cite{Comtet72,Comtet74}. 
Lentin ~\cite{Lentin} showed that $f(n)$, the number of indecomposable permutations (hence the number of connected permutation graphs) of length $n$, satisfies the recurrence relation 
\[	n!-f(n)=\sum_{i=1}^{n-1}(n-i)!f(i), \quad f(1):=1,			\]
and consequently, $f(n)$ is the coefficient of $t^n$ in the series $1- \big( \sum_{k\ge 0}k!t^k\big)^{-1}$. Comtet ~\cite{Comtet72} showed that a permutation chosen uniformly at random from $\sper{n}$, the set of permutations of $[n]$, is indecomposable with probability $1-2/n+O(n^{-2})$. Later, Cori et al.~\cite{CMR} considered the random permutation of $[n]$ with a given number $m$ of cycles and showed that the probability of indecomposability increases from 0 to 1 as $m$ decreases from $n$ to 1. In a more recent paper \cite{AP}, the first author and Pittel 
studied a uniformly random permutation of $[n]$ with a given number $m$ of inversions and found an evolution of this random permutation, where the evolution starts with the identity permutation (no inversions) and reaches the unique permutation with ${n \choose 2}$ inversions, each time gaining one
 inversion. In this evolution, they showed that the probability of indecomposability is monotone increasing in $m$, and they found a threshold value of $m$ around which the transition from `being decomposable' to `being indecomposable' occurs with high probability. Asymptotic sizes of the largest and the smallest components of the corresponding graph were also found when the number of inversions is slightly smaller than the threshold value.

The well-known Cayley's formula states that the number of trees on $n$ vertices is $n^{n-2}$. In this paper, we first find that the number of trees among permutation graphs  on $n$ vertices is $2^{n-2}$. Then we study the tree $T_n$ that is chosen uniformly at random from all these $2^{n-2}$ trees as $n$ tends to infinity.
In particular, we study the degree distribution of $T_n$, the maximum degree in $T_n$, the diameter of $T_n$, and the size of a minimum dominating set in $T_n$. We find that the number of leaves and the diameter are binomially distributed. Denoting by $D_i$ the number of degree-$i$ vertices in $T_n$, we prove that $(D_i)_{i=1}^{m}$ is asymptotically jointly normal for any $m$.
Furthermore, we find the asymptotic distribution of the maximum degree in $T_n$ as $n\to \infty$.
Finally, we show that the size of a minimum dominating set, $\gamma(T_n)$, is also asymptotically normally distributed with mean $n/3+O(1)$ and variance $0.26n+O(1)$. 

The maximum degree and the diameter of a random tree have been studied extensively for various classes of trees. For example, Moon~\cite{Moon} showed that $\Delta_n\log \log n/\log n$ approaches 1 in probability as $n$ tends to infinity, where $\Delta_n$ is the maximum degree of a tree chosen uniformly at random from all the trees on $n$ labeled vertices. For the same random tree, R\'{e}nyi and Szekeres~\cite{RS} showed that the diameter is of order $\Theta(\sqrt{n})$ with probability approaching 1.
Cooper and Zito~\cite{CZ} showed that the size of a minimum dominating set of a random recursive tree is $dn+o(n)$ with probability approaching 1 as $n\to \infty$, where $d\approx 0.3745$.

In Section~\ref{sec: trees} we find the number of trees and the number of forests with a  given number of trees.  In Section~\ref{sec: random trees}, we study the shape of $T_n$, the degree distribution of $T_n$, the maximum degree in $T_n$, the diameter of $T_n$, and the domination number of $T_n$.

\section{Number of Permutation Trees}\label{sec: trees}

Let $\sper{n}$ denote the set of permutations of the set $[n]$. 
In this section we count permutations $\bw \in \sper{n}$ such that $G_{\bw}$ is a tree or a forest. 
In this work, a \emph{permutation tree} refers to a tree that is also a permutation graph and a \emph{tree permutation} refers to a permutation whose graph is a tree. Recall that the number of edges in $G_{\bw}$ is the same as the number of inversions of $\bw$ and $G_{\bw}$ is connected if and only if $\bw$ is indecomposable. 

Let $n>1$ and let $\bw=w_1,\dots,w_n \in \sper{n}$ be a permutation such that $G_{\bw}$ is a tree. Let $m=m(\bw)=n-w_n$. We must have $m>0$ since otherwise $w_n=n$ is an isolated vertex in $G_{\bw}$. Moreover, we have the following simple observations.

\medskip
\noindent \textbf{(i)} The numbers $n-m+1,\dots,n$ appear in increasing order in $\bw$. Otherwise, we have $w_i>w_j>w_n=n-m$ for some $i<j<n$ and hence these three vertices form a triangle in $G_{\bw}$. (Similarly, the numbers $1,\dots,w_1-1$ must be in increasing order.)

\medskip
\noindent \textbf{(ii)} There is no number smaller than $n-m$ appearing after $n-m+2$ in $\bw$. Otherwise, using the previous observation, we have a 4-cycle on the vertices $n-m+1$, $n-m+2$, $j$, and $n-m$ for some $j<n-m$. 

From these observations we get
\begin{equation}\label{n-m}
(w_{n-m+1},\dots,w_{n}) = (n-m+2,\dots,n,n-m).
\end{equation}
Note that the vertices $n-m+2,\dots,n$ are leaves adjacent to $n-m$ in $G_{\bw}$. The only other vertex adjacent to $n-m$ is $n-m+1$, which lies in $\vecw{1}{n-m}$ and replacing $n-m+1$ with $n-m$ in $\vecw{1}{n-m}$, we obtain a tree permutation in $\sper{{n-m}}$. 

We now describe a two-case insertion algorithm to produce tree permutations recursively. For $n=1$, there is a unique permutation, which is a tree permutation. For $n=2$, the only tree permutation is $2,1$. Suppose $n\ge 2$ and let $\bw \in \sper{n}$ be a tree permutation. We produce two tree permutations by inserting the number $n+1$ to $\bw$ in two ways as follows.

\medskip
\noindent \textbf{Insertion algorithm.}
\begin{enumerate}
\item[$I_1$:]\label{op1}  Insert $n+1$ between $w_{n-1}$ and $w_n$ in $\bw$. In this case $n+1$ becomes a leaf adjacent to $w_n$. Denoting the new permutation by $\bw'$, we have
\[
\bw' = w_1,\dots,w_{n-1},n+1,w_n .
\]
This operation increases $m$ by 1, that is, $m(\bw')=m(\bw)+1$.
\item[$I_2$:]\label{op2} Substitute $n$ with $n+1$ in $\bw$ and put the number $n$ at 
 the end of the new permutation $\bw''$. Hence, if $n=w_j$ in $\bw$, then we have
\[
\bw'' = w_1,\dots,w_{j-1},n+1,w_{j+1},\dots,w_n, n .
\]
In this case we have $m(\bw'')=1$.
After this operation, $\{n,n+1\}$ becomes an edge in $G_{\bw''}$ and vertex $n$ loses its neighbors in $G_{\bw}$ to vertex $n+1$ in $G_{\bw''}$. All other adjacency relations are preserved.
\end{enumerate}

Clearly, for any given $\bw$, the pair $(\bw',\bw'')$ is obtained uniquely. It is also easy to see that given a tree permutation $\boldsymbol \tau$ in $\sper{n+1}$, there is a unique tree permutation $\bw$ in $\sper{n}$ such that $\boldsymbol \tau=\bw'$ or $\boldsymbol \tau=\bw''$. If $m(\boldsymbol \tau)>1$, then $\boldsymbol \tau=\bw'$, where $\bw$ is the permutation obtained from $\boldsymbol \tau$ simply by removing $n+1$  (reverse $I_1$). If $m(\boldsymbol \tau)=1$, then $\boldsymbol \tau=\bw''$, where $\bw$ is the permutation obtained from $\boldsymbol\tau$ by first substituting $n+1$ with $n$ and then removing the last element $\tau_{n+1}$ (reverse~$I_2$). Hence,  for $n\ge 3$, there is a bijection between tree permutations of length $n$ and $\{I_1,I_2\}^{n-2}$. Combining this with the fact that there is only one tree permutation of length 1 and one of length 2, we obtain the following result.

\begin{theorem}\label{cor: t_n=}
Let $t_n$ be the number of tree permutations of length $n$. We have $t_1=1$ and $t_n=2^{n-2}$ for $n\ge 2$.
\end{theorem}

\subsection{Shape of a tree permutation}\label{s:shape}

A \emph{caterpillar} is a tree such that all the nonleaves lie on a single path, which we call the \emph{central path}. We will show that if $\bw$ is a tree permutation, then $G_{\bw}$ is a caterpillar.

Let $n\ge 3$ and let $\vecw{1}{n}$ be a tree permutation. Note that $n$ and $w_n$ are adjacent in $G_{\bw}$. 
Also, it follows from~\eqref{n-m} that exactly one of the vertices $n$ and $w_n$ is a leaf. Similarly, $1$ and $w_1$ are adjacent and exactly one of them is a leaf. 

\begin{theorem}\label{caterpillar}
Let $n\ge 3$. Let $\boldsymbol \tau$ be a tree permutation of length $n$ and $G_{\boldsymbol \tau}$ be the corresponding tree. Then $G_{\boldsymbol \tau}$ is a caterpillar. 
If $\tau_1=n$ (respectively $\tau_n=1$), then the central path consists of the unique vertex $n$ (respectively $1$). If $\tau_1 \not =n$ and $\tau_n \not =1$, then
one endpoint of the central path lies in $\{1,\tau_1\}$ and the other one lies in $\{n,\tau_n\}$.
\end{theorem}

\begin{proof}
The proof is by induction on $n$.
It is easy to verify that the claim holds for $n=3$. Now suppose $n>3$ and let $\bt= \tau_1,\dots,\tau_n$ be a tree permutation of length $n$. Let $\bw$ be the unique tree permutation in $\sper{n-1}$ that produces $\boldsymbol \tau$ with one of the operations in the insertion algorithm. By induction hypothesis, $G_{\bw}$ is a caterpillar with the properties given in the statement of the theorem. Let $P_{\bw}$ denote the central path in $G_{\bw}$. 

First suppose that $\boldsymbol \tau$ is produced from $\bw$ with the first operation. Hence, the graph $G_{\bt}$ is obtained from $G_{\bw}$ by adding the vertex $n$ and the edge $\{n,w_{n-1}\}$. We have three cases.
\begin{enumerate}[(i)] \itemsep=0pt \parsep=0pt \partopsep=0pt \topsep=0pt
\item $w_{n-1}=1$. In this case $G_{\bw}$ is a star with the central vertex $1$ and so is $G_{\bt}$.
\item $w_1=n-1$.  In this case, $\bw=n-1,1,\dots,n-2$ and $\bt= n-1,1,\dots,n-3,n,n-2$. Consequently, the only nonleaves in $G_{\tau}$ are $\tau_1=n-1$ and $\tau_{n}=n-2$.
\item $w_{n-1}\not =1$ and $w_{1}\not =n-1$. Here $P_{\bw}$ has at least two vertices and it starts with one of the vertices $1$ or $w_1$.
If $w_{n-1}$ is an end vertex of $P_{\bw}$, then adding the vertex $n$ and the edge $\{n,w_{n-1}\}$ we get another caterpillar with the same central path $P_{\bw}$. If $w_{n-1}$ is not an end vertex of $P_{\bw}$, then by the induction hypothesis, $n-1$ is an end vertex of $P_{\bw}$ and $w_{n-1}$ is a leaf adjacent to $n-1$. Adding the vertex $n$ and the edge $\{n,w_{n-1}\}$ to $G_{\bw}$ we obtain a caterpillar with the central path $P_{\bw}\cup \{n-1,n\}$, so that $n$ becomes an end vertex in the new central path.
\end{enumerate}

The analysis when $\bt$ is produced from $\bw$ with the second operation is very similar and we skip the details.
\end{proof}

\subsection{Forest Permutations}\label{s:forest}

Now we turn our attention to \emph{forest permutations}, which are the permutations whose graphs are acyclic.
Only induced cycles in permutation graphs are triangles and cycles of length 4. These two cycles correspond to patterns of $321$ and $3412$, respectively.
In other words, forest permutations correspond to those permutations avoiding  the patterns  $321$ and $3412$. The sequence enumerating such permutations (along with several other interpretations of these numbers) are given in \cite[Sequence A001519]{sloane}. Namely, denoting the number of length-$n$ permutations avoiding the patterns $321$ and $3412$ by $f_n$, we have the following recurrence relation:
\[
f_n=3f_{n-1}-f_{n-2}; \quad f_1=1, \quad f_2=2.
\]
Solving this recurrence relation we find the number of forest permutations
\[
f_n=\frac{\sqrt 5-1}{2\sqrt5}\left( \frac{3+\sqrt5}{2}\right)^n  +  \frac{\sqrt 5+1}{2\sqrt5}\left( \frac{3-\sqrt5}{2}\right)^n.
\]
Here we find $f(n,m)$, the number of forest permutations with $m$ trees and a total of $n$ vertices.

Let $T(y)$ be the generating function of tree permutations. By Theorem~\ref{cor: t_n=} we have
\[
T(y)=y+\sum_{n\ge 2}2^{n-2}y^n = y+\frac{y^2}{1-2y}.
\]
Note that $f(n,n)=1$ since the identity permutation is the only inversion-free permutation. Let $\bw$ be a permutation and $i<j<k$.
If $w_i>w_k$, then either $w_i>w_j$ or $w_j>w_k$. In other words, if $w_i$ and $w_k$ are neighbors in $G_{\bw}$, then $w_j$ has at least one neighbor in $\{w_i,w_k\}$. It follows from this observation that if $C$ is a connected component of $G_{\bw}$, the vertex set of $C$ is $\{w_a,w_{a+1},\dots,w_b \}$ for some integers $a$ and $b$. 
In other words, a connected component of $G_{\bw}$ consists of consecutive terms in the permutation~$\bw$. Thus, we have
\begin{align*}
f(n,m) &= [y^n]T(y)^m = [y^{n-m}] \left( 1+y(1-2y)^{-1}\right)^m 	\\
&= [y^{n-m}] \sum_{k\ge 1} {m \choose k}y^k(1-2y)^{-k} 		\\
&= \sum_{k\ge 1} {m \choose k} [y^{n-m-k}](1-2y)^{-k}  	\\
&= \sum_{k= 1}^{\min\{m,n-m\}} {m \choose k} {n-m-1 \choose k-1}2^{n-m-k}.
\end{align*}
Hence we have the following theorem.

\begin{theorem}
The number of permutation forests with $n$ vertices and $m$ trees is given by the formula
\[
f(n,m)=
\begin{cases}
1 & \text{ if } n=m,		 \\
\sum_{k= 1}^{\min\{m,n-m\}} {m \choose k} {n-m-1 \choose k-1}2^{n-m-k} & \text{ if } m<n.
\end{cases}
\] 
\end{theorem}

\section{The permutation graph of a random tree permutation}\label{sec: random trees}

Let $\stp{n}$ be the subset of $\sper{n}$ consisting of tree permutations and let $\spt{n}$ denote the set of trees corresponding to these tree permutations, i.e.\ 
\[
\spt{n}= \{G_{\bw}: \bw \in \stp{n} \}.
\]
In Section~\ref{sec: trees}, we proved $|\stp{n}|=|\spt{n}|=2^{n-2}$ for $n\ge 2$.  
In this section we turn $\spt{n}$ into a probability space by equipping it with the uniform probability measure. 
We denote by $T_n$ a random element of $\spt{n}$ and study various graph properties of $T_n$.
In this section we denote by $\bt=\tau_1\dots\tau_n$ the random permutation corresponding to $T_n$, i.e.\ $G_{\bt}=T_n$.

\subsection{The number of leaves}\label{s:leaves}

Let $\leaves{n}$ denote the number of leaves in $T_n$. Here we find the distribution of $L_n$.

\begin{lemma}\label{n or sigman}
Let $n\ge 3$. The vertex $\tau_{n}$ is a leaf in $T_n$ with probability $1/2$. Consequently, the vertex $n$ is a leaf with probability $1/2$.
\end{lemma}

\begin{proof}
Recall from Section~\ref{s:shape} that exactly one of $\{n,\tau_n\}$ is a leaf in $T_n$. Further, the vertex $\tau_n$ is a leaf  in $G_{\bt}$ if and only if $\bt$ is produced from a tree permutation in $\stp{n-1}$ with the second operation in the insertion algorithm. Since $\bt$ is generated by a uniformly random sequence in $\{I_1,I_2\}^{n-2}$, in probabilistic language it means that the last insertion is performed with the second operation with probability $1/2$, which proves the lemma.
\end{proof}

\begin{lemma}\label{leaves}
Let $n\ge 3$. The number of leaves $L_n$ in $T_n$ is distributed as $2+\Bin{n-3}{1/2}$.
\end{lemma}

\begin{proof}
Consider the insertion algorithm applied to $\bw=\vecw{1}{n}$. When the first operation is applied, the vertex $n+1$ becomes a leaf, the  degree of $w_n$ increases by 1, and the other degrees do not change. Hence, if $w_n$ is a leaf in $G_{\bw}$, then the number of leaves stays the same after the operation, and if $w_{n}$ is not a leaf, then the number of leaves increases by 1.

Similarly, when the second operation is applied, the vertex $n$ becomes a leaf in $G_{\bw''}$, the degree of $n+1$ in $G_{\bw''}$ becomes one more than the degree of $n$ in $G_{\bw}$, and all other degrees stay the same. Hence, if $n$ is a leaf in $G_{\bw}$, then the number of leaves stays the same after the second operation, and if $n$ is not a leaf, then the number of leaves increases by 1.

The lemma now follows easily from Lemma~\ref{n or sigman} and induction on $n$.
\end{proof}

\subsection{The diameter}

The diameter of a tree $T$ is the length (number of edges) of the longest path in $T$. In a caterpillar, a longest path starts and ends with leaves and contains all the vertices in the central path. Thus, the diameter of a tree in $\spt{n}$ is two more than the length of the central path. Consequently, denoting the diameter of $T_n$ by $\diam$, 
we have
\begin{equation}\label{diam+leaves}
\diam=n-\leaves{n}+1. 
\end{equation}
Combining~\eqref{diam+leaves} with Lemma~\ref{leaves} gives the next result.

\begin{lemma}
For $n\ge 3$, $\diam$  is distributed as $2+\Bin{n-3}{1/2}$. 
\end{lemma}

\subsection{The highest degree and the number of vertices of a given degree in $T_n$}

In this section we study $D_k=D_k(T_n)$, the number of vertices of degree $k$ in $T_n$. Note that $D_1$ denotes the number of leaves, which was separately studied in Section~\ref{s:leaves}. We start with some observations on $G_{\bw}$ for a tree permutation~$\bw$.

Let $\bw=w_1,\dots,w_n$ be a tree permutation. We say that $w_k$ is a left-to-right maximum if there is no $i \in [k-1]$ such that $w_i>w_k$.
Let $W_1$ be the set of left-to-right maxima and $W_0$ be the rest of the numbers  in $\bw$. 

\begin{lemma}\label{lem:W0W1}
The graph $G_{\bw}$ has the bipartition $(W_0,W_1)$.
\end{lemma}

\begin{proof}
Clearly, two elements $w_i$ and $w_j$ of $W_1$ cannot be neighbors since both of them are left-to-right maxima. Now let $i<j$ and suppose $w_i,w_j \in W_0$. 
Since $w_i \in W_0$, there is some $k$ such that $k<i$ and $w_k>w_i$. Now, we must have $w_i<w_j$ since otherwise $w_k,w_i,w_j$ would form a triangle. Hence $w_iw_j$ is not an edge, which finishes the proof.
\end{proof}

This lemma says that elements of $W_0$ as well as elements of $W_1$ appear in increasing order in $\bw$. 
This fact implies that (i) if $w_k \in W_1$, then the neighbors of $w_k$ lie in the set $\{w_i \in W_0: i>k\}$ and (ii) if $w_k\in W_0$, then the neighbors of $w_k$ lie in the set $\{w_i \in W_1: i<k\}$.
For $k\in [n]$, let $\degree(w_k)$ and $N(w_k)$ denote the degree of $w_k$ and the set of neighbors of $w_k$ in $G_{\bw}$. 
To formulate our result about $N(w_k)$, we define the block decomposition $B_1,\dots,B_{2\ell}$ of $\bw$ as follows: 
\begin{enumerate}[(i)] \topsep=0pt \parsep=0pt \partopsep=0pt \itemsep=0pt
\item each $B_j$ consists of vertices with consecutive indices,
\item the indices of the vertices in $B_j$ are smaller than those of $B_{j+1}$, 
\item $B_{2j-1} \subseteq W_1$ and $B_{2j}\subseteq W_0$ for any integer $j\in [\ell]$.
\end{enumerate}
Note that there must be an even number of blocks since $w_1 \in W_1$ and $w_n \in W_0$. Let $b_j=|B_j|$. For any $j$,  we denote the smallest and largest elements of $B_j$ by $f_j$ and $l_j$, respectively. These are the elements with the smallest and largest indices in $B_j$, respectively, as well.

\begin{lemma}\label{lem:obs}
For $k\in [n]$, the following hold for $\degree(w_k)$ and $N(w_k)$.
\begin{enumerate}[\textup(a\textup)] \topsep=0pt \parsep=0pt \partopsep=0pt \itemsep=0pt
\item If $w_k \in B_{2i-1}$, and $w_k\not = l_{2i-1}$, then $\degree(w_k)=1$ and $N(w_k)= \{f_{2i}\}$.  
\item If $w_k \in B_{2i-1}$, $w_k = l_{2i-1}$, and $B_{2i}$ is not the last block, then $\degree(w_k)=b_{2i}+1$ and $N(w_k)= B_{2i} \cup \{f_{2i+2}\}$. 
\item If $w_k \in B_{2i-1}$, $w_k = l_{2i-1}$, and $B_{2i}$ is the last block, then $\degree(w_k)=b_{2i}$ and $N(w_k)= B_{2i}$. 
\item If $w_k \in B_{2i}$ and $w_k\not = f_{2i}$, then $\degree(w_k)=1$ and $N(w_k)= \{l_{2i-1}\}$.  
\item If $w_k \in B_{2i}$ for some $i\ge 2$ and $w_k = f_{2i}$, then $\degree(w_k)=b_{2i-1}+1$ and $N(w_k)= B_{2i-1}\cup \{l_{2i-3}\}$. 
\item If $w_k \in B_{2}$ and $w_k = f_{2}$, then $\degree(w_k)=b_1$ and $N(w_k)= B_{1}$.   
\end{enumerate}
\end{lemma}

In words, this lemma says the following. 
If $w_k$ is a left-to-right maximum, to find $N(w_k)$, we start reading $\bw$ from $w_{k+1}$ and keep record of all the non-left-to-right maxima until we see the first 
left-to-right maximum followed by a non-left-to-right maximum. (This last non-left-to-right maximum following a left-to-right-maximum is also recorded.)
These recorded vertices will be the neighbors of $w_k$. 
Similarly, if $w_k$ is not a left-to-right maximum, we start reading $\bw$ backwards from $w_{k-1}$ and keep record of all the left-to-right maxima until we see the first non-left-to-right maximum followed by a left-to-right maximum.

\begin{example}
Let $\bs w= 2,5,1,3,6,7,11,4,8,9,10$. Here $W_1=\{w_1,w_2,w_5,w_6,w_7\}$ and $W_0=\{w_3,w_4,w_8,w_9,w_{10},w_{11}\}$. We have $N(w_2)=N(5)=\{w_3,w_4,w_8\}$ and $N(w_8)=N(4)=\{w_2,w_5,w_6,w_7\}$. 
\end{example}

\begin{proof}[Proof of Lemma~\ref{lem:obs}]
We prove only the first three parts as the others follow immediately from the first three parts combined with Lemma~\ref{lem:W0W1}. 

$(a)$ If $\degree(w_k)>1$, then any two neighbors of $w_k$ together with $w_k$ and $w_{k+1}$ form a 4-cycle, which is a contradiction. Hence $\degree(w_k)=1$ and $w_k$ must be a neighbor of the smallest number appearing after $w_k$, which is~$f_{2i}$.

$(b)$ In this case $w_k$ is the largest element of $\cup_{j\le 2i}B_j$, which means that it is a neighbor of each element in $B_{2i}$. Moreover, if $w_k$ is not larger than any element of $\cup_{j>2i}B_j$, then there is no edge from $\cup_{j\le2i}B_j$ to $\cup_{j>2i}B_j$, a contradiction. If $w_k$ is greater than both of $w_a$ and $w_b$ for some $w_a,w_b \in \cup_{j>2i}$, then $w_k$, $w_a$, $w_b$ and any element of $B_{2i+1}$ form a 4-cycle, a contradiction. Thus, $w_k$ is greater than exactly one element in $\cup_{j>2i}$, which is $f_{2i+2}$.

$(c)$ This is similar to $(b)$ but since $B_{2i+2}$ does not exist in this case, all the neighbors of $w_k$ are in $B_{2i}$.
\end{proof}

\begin{corollary}
Let $\bs d= (\degree(w_1),\dots,\degree(w_n))$. If there are $2k$ blocks in the block decomposition of $\bw$, then
\[
\bs d= (\dots, 1^{b_{2i-1}-1},b_{2i}+1-\ind{(i=k)},b_{2i-1}+1-\ind{(i=1)},1^{b_{2i}-1},\dots)
\]
where $\ind{A}$ denotes the indicator of $A$.
\end{corollary}

Since $w_1 \in W_1$ and $w_n \in W_0$, there are at most $2^{n-2}$ pairs $(W_0,W_1)$. 
A pair $(W_0,W_1)$ can be encoded by a vector $(a_1=1, a_2,\dots,a_{n-1},a_n=0) \in \{0,1\}^n$: we have $a_k=1$ if and only if $w_k\in W_1$. 

\begin{lemma}\label{lem:distinctW0W1}
Distinct tree permutations correspond to distinct pairs $(W_0,W_1)$.
\end{lemma}

\begin{proof}
Consider the insertion algorithm given in Section~\ref{sec: trees}. 
Let $\bs a =(1,a_2,\dots,a_{n-1},0)$ be the described encoding of the bipartition of a tree permutation $\bw$ of $[n]$. 
If we insert $n+1$ to $\bw$ via $I_1$, then the new sequence is updated to $(1,a_2,\dots,a_{n-1},1,0)$. If we insert $n+1$ to $\bw$ via $I_2$, then the new sequence is updated to $(1,a_1,\dots,a_{n-1},0,0)$. The only difference between these two updated sequences is on their next-to-last components. Hence, for $k\ge 3$, if $k$ is inserted to the permutation via $I_1$, we have $a_{k-1}=1$, otherwise $a_{k-1}=0$. This finishes the proof of the lemma 
 since a tree permutation is obtained via a unique sequence of $I_1$'s and $I_2$'s.
\end{proof}

In the proof of the following lemma and later, we mean by a \emph{block of a 0-1 sequence} a maximal run of 0's or 1's in the sequence.

\begin{lemma}\label{thm:degrees}
Let $\bs y$ denote a random 0-1 sequence of length $n-2$. Let $Y_i$ denote the number of blocks whose length is equal to $i$ in $\bs y$ and let $Y=\sum Y_i$. Then,
\[
(D_1,D_2,D_3\dots) \stackrel{d}{=} (n-Y,Y_1,Y_2\dots),
\]
where $\stackrel{d}{=}$ means `equal in distribution'.
\end{lemma}

\begin{proof}
By the previous lemma and the encoding of the pairs $(W_0,W_1)$, there is a canonical bijection between 0-1 sequences of length $n-2$ and tree permutations of $[n]$. Let $W_0'=W_0\setminus\{w_n\}$ and $W_1'=W_1\setminus\{w_1\}$ and consider the block decomposition $B_1',\dots,B_k'$ of $w_2,\dots,w_{n-1}$. Unlike the former case (where $w_1$ and $w_n$ were also considered), now we may have $B_1'\subseteq W_0'$, $B_k'\subseteq W_1'$, or $k$ odd. It is easy to verify using Lemma~\ref{lem:obs} that corresponding to each block $B_j'$, there is a unique element $w_k$ of degree $|B_j'|+1$. All the remaining elements are of degree 1. 

Now the lemma follows immediately from coupling the random tree permutation $\bt$ with $\bs y$, i.e.\ by letting $\bt$ be the permutation represented by $\bs y$.
\end{proof}

Let $\hd$ denote the highest degree in $T_n$. An immediate consequence of Lemma~\ref{thm:degrees} is the following.

\begin{corollary}\label{cor:highest}
If $\bs y$ and $Y_i$ are as defined in Lemma~\ref{thm:degrees}, then
\[
\hd \stackrel{d}{=} 1+\max\{i: Y_i>0\}.
\]
In other words, $\hd-1$ is distributed as the size of the largest block in a random 0-1 sequence of length $n-2$.
\end{corollary}

A random 0-1 sequence of length $n$ can be generated by first choosing the first element randomly and then by flipping a coin $n-1$ times. During this flipping process, if a tail ($T$) comes up, then we put the same symbol as the previous one and hence extend the size of the current block by one, and if a head ($H$) comes up, we put a different symbol than the previous one and hence start a new block. For instance, for $n=8$, if the first symbol of the sequence is 0 and the outcome of coin flips is $THHTTHT$, then we have the sequence $00100011$. Note that the number of blocks in the 0-1 sequence obtained this way is one more than the number of heads and the size of the largest block is one more than the longest run of tails. Now let $\tail{n}$ denote the longest run of tails in a sequence of $n$ coin flips. Combining this fact with Corollary~\ref{cor:highest}, we get
\be \label{Hn=}
\hd \stackrel{d}{=} 2+ \tail{n-3}.
\ee

F\"{o}ldes~\cite{Foldes} proved the following result regarding the distribution of $\tail{n}$ (see also earlier,  unpublished work by Boyd~\cite{Boyd}).

\begin{theorem}[F\"{o}ldes]\label{thm:Foldes}
For any integer $k$, we have
\[	\pr(\tail{n}-\lf \log_2n\rf < k) = \exp\lp -2^{-k-1+\{\log_2n\}}\rp +o(1),				\]
where $\{x\}:= x-\lf x \rf$ for any positive real number $x$.
\end{theorem}

Combining Theorem~\ref{thm:Foldes} with~\eqref{Hn=} gives the following corollary.

\begin{corollary}
For any integer $k$, we have
\[ 
\pr(\hd-\lf \log_2(n-3)\rf < k) = \exp\lp -2^{-k+1+\{\log_2(n-3)\}}\rp +o(1).
\]	
\end{corollary}

In view of this discussion and Theorem~\ref{thm:degrees}, for the distribution of the 
degree sequence $(D_i)$,  
it is enough to study the number of blocks of a given size in a random 0-1 sequence of length $n-2$ or the runs of tails in flipping a coin $n-3$ times. 
Let $\bs s=(s_1,\dots,s_{n-3})$ denote a random $H$-$T$ sequence which represents the outcome of $n-3$ coin flips.
Let $\bs y=(y_1,\dots,y_{n-2})$ denote the random 0-1 sequence corresponding to $\bs s$, where $y_1$ is chosen randomly and independently of $\bs s$.
Let $Y_k$ denote the number of blocks of size $k$ in $\bs y$. Each run of {tails} of length $k-1$ corresponds to a block of size $k$ in $\bs y$. Let $Y^*_k$ count the strings of length $k+1$ equal to $HT^{k-1}H$ in $\bs s$. Denoting the sizes of the first and last blocks in $\bs y$ by $b_1$ and $b_{l}$, respectively , we have 
\be \label{Y-Y*}
Y^*_k= Y_k-\ind{(b_1=k)}-\ind{(b_l=k)} \quad \text{and} \quad \sum_{k\ge 2}(Y_k-Y_k^* )\in \{0,1,2\}.
\ee
Letting $\xi_i=\xi_i(k)$ be the indicator of $\{ (s_i,\dots,s_{i+k})=(HT^{k-1}H)\}$ and using the linearity of the expectation, we get
\be \label{meanY*}
\mean{Y_k^*}= \sum_{i=1}^{n-k-3} \mean{\xi_i} = (n-k-3)2^{-k-1}.
\ee
Also, routine calculations yield
\begin{align} \label{varY*}
\var{(Y_k^*)} = \frac{2^{k+1}+1-2k}{2^{2k+2}}n +O(k2^{-k}).
\end{align}

Together with Chebyshev's Inequality, these two equations  imply that $Y_k^*$ (and hence $Y_k$) is concentrated around its expected value, which is roughly $n/2^{k+1}$.  In fact, we have the joint normality of $(Y_1,\dots,Y_m)$  and hence of the degrees $(D_1,\dots,D_m)$ for any fixed $m$.

\begin{theorem}\label{thm:degrees,clt}
For any $m\ge1$, as $n\to\infty$ one has 
\[
\frac1{\sqrt n} \left(Y^*_{k}-\E Y^*_{k}\right)_{k=1}^m \stackrel d\to N(0,\Sigma),
\]
where $\Sigma=[\sigma_{i,j}]$ with 
\[
\sigma_{i,i}=\frac1{2^{i+1}}\left(1-\frac{2i-3}{2^{i+1}}\right),\quad\mbox{and\ }\sigma_{i,j}
=-\frac{i+j-3}{2^{i+j+2}}, \quad i\ne j,\quad i,j=1,\dots,m.
\]  
\end{theorem}

\begin{remark}
Since the joint convergence in $\Bbb R^\infty$ is defined through the joint convergence of any finitely many components, the above theorem can be re--stated as the convergence of the infinite--dimensional vector in $\Bbb R^\infty$
\[\frac1{\sqrt n}
\left(Y^*_{k}-\E Y^*_{k}\right)_{k=1}^\infty \stackrel d\to N(0,\Sigma),\]
with $\sigma_{i,j}$ given above for all $ i,j\ge1$.  
Furthermore, if $G=(G_1,G_2,\dots)$ is a mean--zero Gaussian vector in $\Bbb R^\infty$ with covariance matrix $\Sigma$ 
and $A:\Bbb R^\infty\to\Bbb R^\infty$ is an infinite dimensional matrix then $AG$ is Gaussian with the covariance matrix $A\Sigma A^T$. Applying this with matrix
\[A=\left[\begin{array}{ccccc}-1&-1&-1&-1&\dots\\1&0&0&0&\dots\\
0&1&0&0&\dots\\
0&0&1&0&\dots\\
\dots&\dots&\dots&\dots&\dots
\end{array}\right]\]
we obtain
\end{remark}
\begin{corollary} 
As $n\to\infty$
\[\frac1{\sqrt n}\left(D_k-\E D_k\right)_{k=1}^\infty\stackrel d\to N(0,A\Sigma A^T) \quad\mbox{in\ }\Bbb R^\infty,\]
where $A$ is as above and $\Sigma$ is as in Theorem~\ref{thm:degrees,clt}. In particular, for any $m\ge 1$ 
\[\frac1{\sqrt n}\left(D_k-\E D_k\right)_{k=1}^m\stackrel d\to N(0,(A\Sigma A^T)_{m\times m}),\]
where $(A\Sigma A^T)_{m\times m}$ is an $m\times m$ northwest corner of $A\Sigma A^T$.
\end{corollary}

\begin{proof}[Proof of Theorem~\ref{thm:degrees,clt}]

We will apply the  following Hoeffding--Robbins Central Limit Theorem~\cite{HR}. 
There are stronger versions of this theorem, see e.g. \cite{RW} and references therein, but the original version of Hoeffding--Robbins is enough for our purpose. 
Interestingly, while most of the later papers concentrate on 1--dimensional random variables, Hoeffding and Robbins actually give a version for random vectors. 
To be precise they state and prove the vector--valued version  for stationary sequences of 2--dimensional, $m$--dependent random vectors (see Theorem~3 in \cite{HR}) but state after the proof that \lq\lq The extension of Theorem~3 to the case $N>2$, as well as to the non--stationary case, is evident and will be left to the reader.\rq\rq \    
We recall that a sequence $(X_n)$ of random variables is $m$--dependent if  for all $k$ and $l$ 
in $\Bbb N$, $(X_1,\dots,X_k)$ and $(X_{k+n},\dots,X_{k+n+l})$ are independent whenever $n>m$, and it is stationary if, for any $j \in \Bbb N$, the distribution of random vector $(X_n,X_{n+1},\dots,X_{n+j})$ does not depend on $n$. 

  \begin{theorem}\label{h-r_thm} (Hoeffding--Robbins) Let $(X_{i,1},X_{i,2},\dots,X_{i,N})$, $i=1,2,\dots$ be a stationary and  $m$--dependent sequence of random vectors in $\Bbb R^N$ such that 
\[
\E X_{1,k}=0,\quad \E|X_{1,k}|^3<\infty,\quad  k=1,\dots,N.
\]
Then as $n\to\infty$ the random vector
\begin{equation}\label{clt}
\frac1{\sqrt n}\left(\sum_{i=1}^nX_{i,1},\dots,\sum_{i=1}^nX_{i,N}\right),
\end{equation}
has a limiting normal distribution with mean zero and covariance matrix $\Sigma=[\sigma_{j,k}]$,
where
\begin{equation}\label{cov}\sigma_{j,k}=\E X_{1,j}X_{1,k}+\sum_{l=1}^m\big(\E X_{1,j}X_{l+1,k}+\E X_{l+1,j}X_{1,k}\big)\quad j,k=1,\dots,N.
\end{equation}

\end{theorem}

We wish to apply this theorem to the random vectors 
\[(\xi_i(1)-\E\xi_i(1),\xi_i(2)-\E\xi_i(2),\dots,\xi_i(N)-\E\xi_i(N)),\quad i=1,2,\dots \] 
with $N=m$. A minor nuisance is that unless the $\xi_i(k)$ are defined based on the \emph{infinite} sequence of coin tosses they are not stationary (this is because in $n$ tosses,   $\xi_i(k)$'s are 0 for $i>n-k$ and thus their distribution is different from that of  $\xi_i(k)$ for $1\le i\le n-k$). We will deal with this issue later and for now we assume that the $\xi_i(k)$ are defined based on an infinite sequence of coin tosses. Then the sequence $(X_{i,1},\dots,X_{i,m})$  is  stationary. Note also that   for every $k\ge1$, $\xi_i(k)$, $i\ge 1$, are $k$-dependent because $\xi_i(k)$ involves positions $(s_i,\dots,s_{i+k})$. Therefore, the random vectors $(\xi_i(1),\dots,\xi_i(m))$, $i\ge1$,  
are $m$--dependent and hence $(X_{i,1},\dots,X_{i,m})$ are $m$--dependent, too. 
Since $\xi_i(k)$ are indicator random variables, it is evident that  $\E|X_{i,k}|^3\le2^3$. Therefore, the asymptotic normality  \eqref{clt} holds and it remains to evaluate the covariance matrix \eqref{cov}.

Since $\xi_{1}(j)$ and $\xi_{l+1}(k)$ involve positions $(s_{1},\dots,s_{j+1})$ and $(s_{l+1},\dots,s_{l+1+k})$, respectively, they are independent if $l>j$, impossible to happen simultaneously if  $l<j$,  and correspond to 
$(s_1,\dots,s_{k+l+1})=HT^{k-1}HT^{l-1}H$  if $l=j$. 
Hence, for $l\ge1$,
\[
\E X_{1,j}X_{l+1,k}=\operatorname{cov}(\xi_{1}(j),\xi_{l+1}(k))=\left\{
\begin{array}{ll}0,&\mbox{if $l>j$}; \\
2^{-j-k-1}-2^{-j-k-2}=2^{-j-k-2},&\mbox{if $l=j$};\\
-2^{-j-k-2},&\mbox{if $l<j$}.
\end{array}\right.
\]
Consequently,
\[
\sum_{l=1}^m\left(\E X_{1,j}X_{l+1,k}+\E X_{l+1,j}X_{1,k}\right)
=-\frac{j-1}{2^{k+j+2}}+\frac1{2^{k+j+2}}-\frac{k-1}{2^{k+j+2}}+\frac1{2^{k+j+2}}
=-\frac{j+k-4}{2^{j+k+2}}.
\]
This holds regardless of whether $j=k$ or not. However, 
\[\E X_{1,j}X_{1,k}=\operatorname{cov}(\xi_1(j),\xi_1(k))=\left\{\begin{array}{ll}
2^{-j-1}-2^{-2(j+1)},&\mbox{if $k=j$};\\
-2^{-j-k-2},&\mbox{if $k\ne j$}.\end{array}\right.\]
Hence,
\[\sigma_{j,j}=\frac1{2^{j+1}}-\frac{2j-3}{2^{2(j+1)}}=\frac1{2^{j+1}}\left(1-\frac{2j-3}{2^{j+1}}\right),\quad \sigma_{j,k}=-\frac{j+k-3}{2^{2(j+1)}},\quad j\ne k.
\]
This proves the central limit theorem in the case of  the infinite number of coin tosses. In Theorem~\ref{thm:degrees,clt} we formally have a triangular array of random vectors
\[
\left(X_{i,1}^{(n)},\dots,X_{i,m}^{(n)}\right)=\left(\xi^{(n)}_i(1)-\E\xi^{(n)}_i(1),\dots,\xi^{(n)}_i(m)-\E\xi^{(n)}_i(m)\right),\quad i=1,2,\dots,n,\quad n\ge 1,
\]
where $\xi_i^{(n)}(k)$ is the indicator of the event $\{(s_i,\dots,s_{i+k})=(HT^{k-1}H)\}$ and $(s_1,\dots,s_n)$ is a sequence of the first $n$ tosses in an infinite sequence of  a coin toss.  
But for $k\le m$, this does not affect the distribution of $(X_{i,k}^{(n)})$ as long as $i\le n-m$. Thus
\[
\frac1{\sqrt n}\sum_{i=1}^nX_{i,k}^{(n)}=\frac1{\sqrt n}\sum_{i=1}^nX_{i,k}+\frac1{\sqrt n}\sum_{i=n-m+1}^n(X_{i,k}^{(n)}-X_{i,k}).
\]
Since $m$ is fixed and $X_{i,k}^{(n)}-X_{i,k}$ are uniformly bounded it follows that 
\[\frac1{\sqrt n}\sum_{i=n-m+1}^n(X_{i,k}^{(n)}-X_{i,k})\stackrel P\to0,\quad \mbox{as } n\to\infty\]
and hence Theorem~\ref{thm:degrees,clt} is proved.
\end{proof}

\subsection{Domination number of $T_n$}

A \emph{dominating set} of a graph $G$ is a subset $S$ of the vertex set of $G$ such that each edge of $G$ is incident to a vertex in $S$. The \emph{domination number} of $G$, denoted $\gamma(G)$, is the minimum size of a dominating set, i.e.
\[
\gamma(G) = \min \{\,|S|: S \text{ is a dominating set of } G\}.
\]

Let $T$ be a tree whose vertices are labeled with integers. A smallest dominating set $S$ of $T$ can be found with the following recursive algorithm, similar to the one   given by Cockayne, Goodman, and Hedetniemi~\cite{CGH}.
 As long as a tree contains at least three vertices, first mark (simultaneously) the neighbors of the leaves of the tree and then delete all the edges incident to these marked vertices. 
Repeat this process as long as there is a tree of size at least 3. 
At the end, we end up with trees of size 1 and 2. 
At this point, mark the vertex with smaller value in each tree of size 2. Finally, put the marked vertices into $S$. 

Now suppose that $T$ is a caterpillar, in which case, the dominating set $S$ produced by the algorithm we described above is a subset of the central path.  Let $S_1 \subset S$ be the set of marked vertices 
produced by 
the first iteration of the algorithm, that is, the set of neighbors of the leaves. Note that $S_1$ contains all the endpoints of the central path as well as all vertices of degree at least $3$, but nothing more.
Let $S_1=\{v_1,\dots,v_k\}$, where the unique path between $v_i$ and $v_{i+1}$ does not contain any other vertex from $S_1$ for $1\le i\le k-1$.
The remaining vertices (if any) of the central path are of degree-2  and they are scattered between the vertices of $S_1$. If there are $n_i$ internal vertices of degree-2 on the path between $v_i$ and $v_{i+1}$, then
\be \label{gammaT}
\gamma(T)= k+\sum_{i=1}^{k-1} \lf n_i/2 \rf.
\ee

In view of this discussion and the adjacency relation given in Lemma~\ref{lem:obs}, to find an asymptotic distribution of  $\gamma(T_n)$, we need to analyze the
structure of the block decomposition of $\bt=\tau_1,\dots,\tau_n$. In particular, $S_1$ differs by at most 2 from the number of vertices of degree at least 3, which is given by $\sum_{k\ge 2} Y_k$ by Lemma~\ref{thm:degrees}. 
The rest of the smallest dominating set described above consists of degree-2 vertices. For this part, (i.e.\ degree-2 vertices in the dominating set) we need to analyze the
maximal runs of blocks of size 1 in the block decomposition of $\bt=\tau_1,\dots,\tau_n$. 
More specifically, letting $B_0=\emptyset$, it follows from Lemma~\ref{lem:obs} that for each pair $i\ge 0$ and $j\ge 1$ such that
\be \label{i,j pairs}
b_i\not = 1, \quad b_{i+1}=\cdots=b_{i+j}=1,  \quad b_{i+j+1}\not = 1,
\ee
where $b_t$ denotes the size of $B_t$, there is a unique set of $j$ vertices, each of degree 2, that connect two vertices of a longest path in $T_n$. Necessarily, a longest path in a caterpillar consists of the central path and two leaves.

In fact, all we need is the information of runs of single-vertex blocks in the block decomposition of $\tau_2,\dots,\tau_{n-1}$. 
Recalling that $W_0'=W_0\setminus \{w_n\}$ and $W_1'=W_1\setminus \{w_1\}$, and using the bijection between the pairs $(W_0',W_1')$ and 0-1 sequences of length $n-2$, it is enough to study the block decomposition of 0-1 sequences. 
Again we will couple the random tree permutation $\tau_1,\dots,\tau_{n}$ with $\bs y$, a random 0-1 sequence of length $n-2$ and we will generate $\bs y$ with the coin-flip algorithm. 
For the runs of blocks of size one, we now need to analyze runs of heads in the random sequence~$\bs s$, where $\bs s$ represents the output of $n-3$ coin flips.

Analogous to $Y_k^*$, we define $Z_k^*$ as the number of strings equal to $TH^{k+1}T$ in $\bs s$. (Note that this string is of size $k+3$ as opposed to $k+1$ in the case of $Y_k^*$.) 
Each such string adds 1 to a run of $k$-consecutive blocks of size one in the block decomposition of $\bt$. 
Note that $Z_k^*$ has the same distribution as $Y_{k+2}^*$.

\begin{theorem}\label{thm:domin}
As $n \to \infty$,
\[
\frac{\gamma(T_n)-\E\gamma(T_n)}{\sqrt n}\stackrel d\to N(0,\sigma^2),
\]
where $\mean{\gamma(T_n)}=n/3+O(1)$ and $\sigma^2=\frac{13}{50}$.
\end{theorem}

\begin{proof}
It follows from~\eqref{gammaT} that
\[
\gamma(T_n)=\sum_{i\ge 2}Y_i+ \sum_{k\ge 2} \lf k/2\rf Z_k^* + \sum_{k\ge 1}\lc k/2\rc ( \ind{A_k}+\ind{B_k}),
\]
where $A_k$ and $B_k$ denote the events that $\bs s$ starts with $H^kT$ and ends with $TH^k$, respectively. (Note that in the last sum we have the ceiling of $k/2$ because if the central path starts (ends) with a degree-2 vertex, this vertex is included in $S$ by the algorithm.)  Since $\pr(A_k)=\pr(B_k)=1/2^{k+1}$, the last sum above is bounded in probability. Using this and~\eqref{Y-Y*}, for the asymptotic distribution of $\gamma(T_n)$, we may only consider the sum
\[
\sum_{k\ge2}\Big(\lfloor k/2\rfloor Z_k^*+Y_k^*\Big).
\]
Let
\[
V_k:=\lfloor k/2\rfloor Z_k^*+Y_k^*-\E(\lfloor k/2\rfloor Z_k^*+Y_k^*), \quad k\ge 2.
\]
For $x\in\Bbb R$, any $\eps>0$ and $m$ to be chosen later we have
\[
\P\Big(\frac1{\sqrt n}\sum_{k\ge 2} V_k\le x\Big)
\le \P\Big(\frac1{\sqrt n}\sum_{k=1}^mV_{k}\le x+\eps,\frac1{\sqrt n} \Big|\sum_{k>m}V_k \Big| \le\eps \Big)
+\P\Big(\frac1{\sqrt n}\Big|\sum_{k>m}V_k\Big|>\eps \Big).
\]
Also,
\begin{align*}
\P\Big(\frac1{\sqrt n}\sum_{k\ge 2} V_k\le x \Big)
&\ge\P\Big(\frac1{\sqrt n}\sum_{k=1}^mV_{k}\le x,\frac1{\sqrt n}\Big|\sum_{k>m}V_k\Big|\le\eps \Big)
\ge \P\Big(\frac1{\sqrt n}\sum_{k=1}^mV_{k}\le x-\eps,\frac1{\sqrt n}\Big|\sum_{k>m}V_k\Big|\le\eps \Big) 	\\
&\quad \ge \P \Big(\frac1{\sqrt n}\Big|\sum_{k\le m}V_k\Big|\le x-\eps \Big)-
\P\Big(\frac1{\sqrt n}\Big|\sum_{k>m}V_k \Big|>\eps \Big).
\end{align*}
Using the inequality $\var(\sum_{k>m}V_k)\le (\sum_{k>m}\sqrt{\var{V_k}})^2$ and the fact that 
$\var(V_k)=O(n/2^k)$ uniformly in $k$, (see~\eqref{varY*}),  
by Chebyshev's inequality we get
\[
\P \Big(\frac1{\sqrt n} \Big| \sum_{k>m}V_k \Big|>\eps \Big)\le\frac1{\eps^2n}\var\Big(\sum_{k>m}V_k \Big)
\le \frac {Cn}{\eps^2n2^{m/2}}=\frac C{\eps^22^{m/2}},
\]
for an absolute constant $C$. Pick $m$ so that $C/2^{m/2}<\eps^3/2$  
and consider 
\[\frac1{\sqrt n}\sum_{k\le m} V_k.\]
We let $\eta_i(k)$ be the indicator of the event $\{(s_i,\dots,s_{i+k+2})=TH^{k+1}T\}$ and write
\[\sum_{k=2}^m\left(\lfloor k/2\rfloor Z_k^*+Y_k^*\right)=\sum_{k=2}^m\sum_{i=1}^n(\lfloor k/2\rfloor\eta_i(k)+\xi_i(k))
=\sum_{i=1}^n\left(\sum_{k=2}^m(\lfloor k/2\rfloor\eta_i(k)+\xi_i(k))\right).
\]
Just as in the proof of Theorem~\ref{thm:degrees,clt}, random variables $\left\{\sum_{k=2}^m\lfloor k/2\rfloor\eta_i(k)+\xi_i(k)\right\}_{i\ge 1}$ 
are $(m+2)$--dependent. They are also stationary (if based on an infinite number of coin tosses).  So, when centered and normalized by $\sqrt n$
 they satisfy the CLT.  Therefore,
\[
\P(N(0,\sigma^2)\le x-\eps)-\eps
\le \liminf_n\P\Big(\frac1{\sqrt n}\sum_{k\ge2}V_k\le x\Big)
\le \limsup_n\P\Big(\frac1{\sqrt n}\sum_{k\ge2}V_k\le x\Big)
\le \P(N(0,\sigma^2)\le x+\eps)+\eps.
\]
The proof of asymptotic normality is completed by letting $\eps\to0$  and dealing with the infinite versus finite sequence of tosses issue in the same way as in the proof of Theorem~\ref{thm:degrees,clt}. We close by computing the expected value and the variance of $\gamma(T_n)$. By~\eqref{meanY*} we already know  that $\E Y^*_k=(n-k-1)/2^{k+1}$ so that 
\[
\sum_{k=2}^\infty\E Y^*_k=\frac n4+  O(1).                    
\]
Since $Z_k^*$ has the same distribution as $Y^*_{k+2}$, we get 
\[
\sum_{k\ge2}\lfloor k/2\rfloor\E Z^*_k= n\sum_{k\ge1}k\left(\frac1{2^{2k+3}}+\frac1{2^{2k+4}}\right)+O(1)
=\frac n{2^3\cdot4}\sum_{k\ge1}\frac k{4^{k-1}}\left(1+\frac12\right)+O(1)
=\frac n{12}+O(1)
\]
and hence $\E \gamma(T_n)= n/3+O(1)$. Furthermore,
\begin{align}\label{varsumVk}
&\var\Big(\sum_{k\ge2}(\lfloor k/2\rfloor Z^*_k+Y^*_k)\Big)=\var\Big(\sum_{i=1}^n\sum_{k\ge2}
(\lfloor k/2\rfloor\eta_i(k)+\xi_i(k))\Big)		\notag 	\\&\quad
=\sum_{i=1}^n\var\Big(\sum_{k\ge2}
(\lfloor k/2\rfloor\eta_i(k)+\xi_i(k))\Big)+2
\sum_{i<j}\operatorname{cov}\Big(\sum_{k\ge2}(\lfloor k/2\rfloor\eta_i(k)+\xi_i(k)),\sum_{k\ge2}(\lfloor k/2\rfloor\eta_j(k)+\xi_j(k))\Big).
\end{align}
Now, for $1\le i\le n-k-2$  
\begin{align*}
&\var\Big(\sum_{k\ge2} (\lfloor k/2\rfloor\eta_i(k)+\xi_i(k))\Big) 
= \sum_{k\ge2}\left\{\lfloor k/2\rfloor^2\var(\eta_i(k))+\var(\xi_i(k))+2\lfloor k/2\rfloor\cov(\eta_i(k),\xi_i(k))\right\}\\&\quad+2\sum_{k<l}\left\{\lfloor k/2\rfloor\lfloor l/2\rfloor\cov(\eta_i(k),\eta_i(l))+\cov(\xi_i(k),\xi_i(l))+\lfloor k/2\rfloor\cov(\eta_i(k),\xi_i(l))+\lfloor l/2\rfloor\cov(\xi_i(k),\eta_i(l))\right\}\\
&= \sum_{k\ge2}\left\{\frac{\lfloor k/2\rfloor^2}{2^{k+3}}\left(1-\frac1{2^{k+3}}\right)+\frac1{2^{k+1}}\left(1-\frac1{2^{k+1}}\right)-
2\frac{\lfloor k/2\rfloor}{2^{k+3}2^{k+1}}\right\}\\&\quad-2\sum_{k<l}\left\{\frac{\lfloor k/2\rfloor\lfloor l/2\rfloor}{2^{k+l+6}}+\frac1{2^{k+l+2}}+\frac{\lfloor k/2\rfloor+\lfloor l/2\rfloor}{2^{l+3+k+1}}\right\} 	
=\sum_{k\ge2}\left(\frac{\lfloor k/2\rfloor^2}{2^{k+3}}+\frac1{2^{k+1}}\right)-\left(\sum_{k\ge2}\frac{\lfloor k/2\rfloor}{2^{k+3}}+\frac1{2^{k+1}}\right)^2\\
&=\frac5{4\cdot3^2}+\frac14-\frac1{3^2}=\frac5{18},
\end{align*}
by the same calculation as in the proof of Theorem~\ref{thm:degrees,clt}.
To compute the second term in~\eqref{varsumVk}, we first use
\begin{align*}&
\operatorname{cov}\Big(\sum_{k\ge2}\big(\lfloor k/2\rfloor\eta_i(k)+\xi_i(k)\big),   \sum_{k\ge2}\big(\lfloor k/2\rfloor\eta_j(k)+\xi_j(k)\big)\Big)	\\&\quad=
\sum_{k\ge2}\Big(\lfloor k/2\rfloor^2\cov(\eta_i(k),\eta_j(k))+\lfloor k/2\rfloor \Big( \cov(\eta_i(k),\xi_j(k))+\cov(\xi_i(k),\eta_j(k))\Big)+\cov(\xi_i(k),\xi_j(k))\Big)	\\&\quad
+2\sum_{k<l}\Big(\lfloor k/2\rfloor\lfloor l/2\rfloor\cov(\eta_i(k),\eta_j(l))+\lfloor k/2\rfloor\cov(\eta_i(k),\xi_j(l))+\lfloor l/2\rfloor\cov(\xi_i(k),\eta_j(l))+\cov(\xi_i(k),\xi_j(l))\Big)
\end{align*}
and then
note that for $i<j$ and $k\le l$ we have 
\begin{align*}\operatorname{cov}(\xi_{i}(k),\xi_{j}(l))&=\left\{\begin{array}{ll}0,&\mbox{if $j>i+k$}; \\
-2^{-k-1}2^{-l-1},&\mbox{if $i<j<i+k$};\\
2^{-k-l-1}-2^{-k-l-2},&\mbox{if $j=i+k$}.\end{array}\right.\\ \operatorname{cov}(\eta_{i}(k),\eta_{j}(l))&=\left\{\begin{array}{ll}0,&\mbox{if $j>i+k+2$}; \\
-2^{-k-3}2^{-l-3},&\mbox{if $i<j<i+k+2$};\\
2^{-k-l-5}-2^{-k-3}2^{-l-3},&\mbox{if $j=i+k+2$}.\end{array}\right.\\
\operatorname{cov}(\eta_{i}(k),\xi_{j}(l))&=\left\{\begin{array}{ll}0,&\mbox{if $j>i+k+2$}; \\
2^{-k-3-l+1}-2^{-k-3}2^{-l-1},&\mbox{if $j=i+k+1$};\\
-2^{-k-3}2^{-l-1},&\mbox{if $i<j<i+k+1$ or $j=i+k+2$},\end{array}\right.\\ \operatorname{cov}(\xi_{i}(k),\eta_{j}(l))&=\left\{\begin{array}{ll}0,&\mbox{if $j>i+k$}; \\
2^{-k+1-l-3}-2^{-k-1}2^{-l-3},&\mbox{if $j=i+k-1$};\\
-2^{-k-1}2^{-l-3},&\mbox{if $i<j<i+k-1$ or $j=i+k$}.\end{array}\right.\ . \end{align*}
Consequently, 
\begin{align*}&
2\sum_{i<j}\sum_{k\ge2}\lfloor k/2\rfloor^2\cov(\eta_i(k),\eta_j(k))=
2\sum_{k\ge2}\left(\lfloor k/2\rfloor^2\sum_i\left(-\frac{k+1}{2^{2(k+3)}}+\frac1{2^{2(k+3)}}\right)
\right)\\&\qquad\sim-\frac{2n}{2^6}\sum_{k\ge2}\frac{k\lfloor k/2\rfloor^2
}{4^{k}}=-\frac n{2^5}\sum_{m\ge1}\left(\frac{m^2(2m)}{4^{2m}}+\frac{m^2(2m+1)}{4^{2m+1}}\right)=-\frac{77}{9000}n.
\end{align*}
Similarly,
\begin{align*}&2\sum_{i<j}\sum_{k\ge2}\lfloor k/2\rfloor\cov(\eta_i(k),\xi_j(k))\sim-\frac n{2^5}\sum_{m\ge1}\frac{m(10m-9)}{4^{2m}}=-\frac 7{1350}n 
\\&2\sum_{i<j}\sum_{k\ge2}\lfloor k/2\rfloor\cov(\xi_i(k),\eta_j(k))\sim-\frac n{2^5}\sum_{m\ge1}\frac{m(10m-19)}{4^{2m}}=\frac{23}{1350}n
\\&2\sum_{i<j}\sum_{k\ge2}
\cov(\xi_i(k),\xi_j(k))\sim-\frac n2\sum_{k\ge2}\frac{k-2}{2^{2k}}=-\frac n{72}
\\&4\sum_{i<j}\sum_{k\le l}\cov(\xi_i(k),\xi_j(l))=4\sum_{k<l}\sum_i\left(-\frac{k-1}{2^{k+l+2}}+\frac1{2^{k+l+2}}\right)
\sim-4n\sum_{k\ge2}\frac{k-2}{2^{k+2}}\sum_{l>k}\frac1{2^l}=-\frac n{36} 
\\&4\sum_{i<j}\sum_{k\le l}\lfloor k/2\rfloor\lfloor l/2\rfloor\cov(\eta_i(k),\eta_j(l))\sim-4n\sum_{k\ge2}\frac{k\lfloor k/2\rfloor}{2^{k+6}}\sum_{l>k}\frac{\lfloor l/2\rfloor}{2^l}=-\frac{43}{1500}n 
\\&4\sum_{i<j}\sum_{k\le l}\lfloor k/2\rfloor\cov(\eta_i(k),\xi_j(l))\sim-\frac n4\sum_{k\ge2}\frac{\lfloor k/2\rfloor(k-2)}{2^k}\sum_{l>k}\frac1{2^l}\sim-\frac7{675}n
\\&4\sum_{i<j}\sum_{k\le l}\lfloor l/2\rfloor\cov(\xi_i(k),\eta_j(l))\sim-\frac n4\sum_{k\ge2}\frac{k-4}{2^k}\sum_{l>k}\frac{\lfloor l/2\rfloor}{2^l} 
\sim\frac{161}{2700}n.
\end{align*}
Combining all of these calculations we finally get
\[
\var\Big(\sum_{k\ge2}\left(\lfloor k/2\rfloor Z_k^*+Y_k^*\right)\Big)\sim n\left(\frac58-\frac{77}{9000}-\frac7{1350}+\frac{23}{1350}-\frac1{72}-\frac1{36}-\frac{43}{1500}-\frac7{675}+\frac{161}{2700}\right)=\frac{13}{50}n
\]
which completes the proof of Theorem~\ref{thm:domin}.
\end{proof}

\section{Concluding Remarks}
Runs of patterns  in $0-1$ sequences have many other connections. For example, they can be used to construct threshold graphs (see e.g. \cite{ChH, MP}) or represent random compositions of integers (see \cite{and} or e.g. \cite{HL,HS} for probabilistic interpretation). Thus, our results are directly applicable to such situations. In particular, Theorem~\ref{thm:degrees,clt} gives the joint asymptotic normality of multiplicities of part sizes in random compositions of an integer $n$, as $n\to\infty$.  Random compositions are often studied in conjunction with samples $(\G_1,\dots,\G_n)$ of iid $\operatorname{Geom}(p)$ random variables. For example,  Grabner, Knopfmacher, and Prodinger~\cite{GKP}  considered runs (of the same values) in such samples and using arguments based 
on generating functions derived, among other things,  the expressions for the expected value, the variance, and the limiting distribution of the number of runs in such samples (see, Propositions~1, Proposition~2, and Theorem~2, respectively in \cite{GKP}). 
We would like to mention that these results are also 
available by probabilistic arguments we used in this paper. Note that the number of runs is given by
\[
R_n=1+\sum_{j=1}^{n-1}I_j,\quad \mbox{where}\quad I_j:=\ind{\G_{j+1}\ne\G_j}.
\]
Since $\G_j$'s are iid and $I_j$ involves only $\G_j$ and $\G_{j+1}$ it follows that $(I_1,\dots,I_{n-1})$ are identically  distributed 1-dependent random variables.
Hence, we can recover the results of Grabner, Knopfmacher, and Prodinger:
\begin{eqnarray*}\E R_n&=&1+(n-1)\P(\G_2\ne\G_1)=1+(n-1)\frac{2q}{1+q}=\frac{2q}{1+q}n+\frac{1-q}{1+q},\\ \var{(R_n)}&=&(n-1)\var{(I_1)}+2(n-2)\operatorname{cov}(I_1,I_2)\\&=&(n-1)\frac{2q}{1+q}\left(1-\frac{2q}{1+q}\right)+2(n-2)\left(\E I_1I_2-\left(\frac{2q}{1+q}\right)^2\right)\\&=&
(n-1)\frac{2q(1-q)}{(1+q)^2}+2(n-2)\frac{q(1-q)^3}{(1+q)^2(1-q^3)}\\
&=&\frac{2q(1-q)^2(2+q^2)}{(1+q)^2(1-q^3)}n-\frac{2q(1-q^2)(3-q+q^2)}{(1+q)^2(1-q^3)},
\end{eqnarray*}
where in the penultimate step we have used the fact that
\begin{eqnarray*}
\E I_1I_2&=&\P(\G_1\ne\G_2\ne\G_3)=\sum_{j=1}^\infty q^{j-1}p(1-q^{j-1}p)^2=1-2\frac{p^2}{1-q^2}+\frac{p^3}{1-q^3}\\&=&\frac{q(1-q)(4q^2+q+1)}{(1+q)(1-q^3)}.
\end{eqnarray*}
Furthermore, $(R_n-\E R_n)/\sqrt n$ is asymptotically normal by  a special case of Theorem~\ref{h-r_thm} (or Theorems~1 or~2 in \cite{HR}).

\section*{Acknowledgement}
Pawe{\l} Hitczenko was partially supported by a Simons Foundation grant \#208766. 
H\"{u}seyin Acan is supported by the National Science Foundation under Award No.~1502650.
 The authors are thankful to anonymous referees for their valuable suggestions.

\end{document}